\documentclass{amsart}
\usepackage{amssymb}
\usepackage{graphicx}

\renewcommand{\bar}[1]{\overline{#1}}
\newcommand{\double}[1]{\widehat{#1}} 

\newcommand\Ad{\mathrm{Ad}}

\newcommand{\eps}{\varepsilon}

\newcommand{\rest}[1]{\big\rvert_{#1}} 

\newcommand{\df}[1]{\mathfrak{#1}}

\newcommand{\lrpar}[1]{\left( #1 \right)}

\renewcommand{\bar}[1]{\overline{#1}}

\usepackage[all]{xy}
\usepackage{eurosym, mathrsfs}

\newcommand\boxb[1]{\square_b}

\newcommand\wt{\widetilde}

\newcommand\ff{\operatorname{ff}}

\newcommand\paperintro%
        {%
         }
\newcommand\paperbody%
        {%
         }

\newtheorem{theorem}{Theorem}

\newtheorem{corollary}[theorem]{Corollary}

\newtheorem{lemma}[theorem]{Lemma}
\newtheorem{proposition}[theorem]{Proposition}

\numberwithin{equation}{section}
\numberwithin{theorem}{section}

\newtheorem{non-theorem}{Non-Theorem}

\theoremstyle{remark}
\newtheorem{definition}[theorem]{Definition}
\newtheorem{remark}[theorem]{Remark}

\newcommand\bV{\mathcal{V}_{\operatorname{b}}}

\newcommand\cFTs{{}^{\Phi}\overline{T}\kern-1pt{}^*}

\newcommand\cB{\mathcal{B}}

\newcommand\cF{\mathcal F}
\newcommand\cG{\mathcal G}

\newcommand\cH{\mathcal{H}}

\newcommand\cI{\mathcal{I}}

\newcommand\cJ{\mathcal{J}}

\newcommand\cL{\mathcal{L}}

\newcommand\cM{\mathcal{M}}

\newcommand\cS{\mathcal S}

\newcommand\cU{\mathcal{U}}

\newcommand\bbR{\mathbb R}
\newcommand\bbS{\mathbb S}

\newcommand\bbZ{\mathbb Z}

\newcommand\CI{{\mathcal{C}}^{\infty}}

\newcommand\CO{{\mathcal{C}}^{0}}

\newcommand\cFNs{{}^{\Phi}\overline N\kern-1pt{}^*}

\newcommand\Id{\operatorname{Id}}

\newcommand\ci{${\mathcal{C}}^\infty$}
\newcommand\clos{\operatorname{cl}}
\newcommand\codim{\operatorname{codim}}

\newcommand\supp{\operatorname{supp}}

\newcommand\Mand{\text{ and }}

\newcommand\Mforall{\text{ for all }}

\newcommand\Mforevery{\text{ for every }}

\newcommand\Mif{\text{ if }}

\newcommand\Min{\text{ in }}

\newcommand\Mnear{\text{ near }}

\newcommand\Mon{\text{ on }}
\newcommand\Mor{\text{ or }}

\newcommand\Mst{\text{ s.t. }}

\newcommand\tb{\operatorname{tb}}

\begin{document}
\title{Resolution of smooth group actions}
\author{Pierre Albin and Richard Melrose}
\address{Department of Mathematics, Massachusetts Institute of Technology\newline
current address: Institut de Math\'ematiques de Jussieu}
\email{albin@math.jussieu.fr}
\address{Department of Mathematics, Massachusetts Institute of Technology}
\email{rbm@math.mit.edu}
\thanks{The first author was partially supported by an NSF postdoctoral
  fellowship and NSF grant DMS-0635607002 and the second author received 
  partial support under NSF grant DMS-1005944.}

\begin{abstract} A refined form of the `Folk Theorem' that a smooth action
  by a compact Lie group can be (canonically) resolved, by iterated blow
  up, to have unique isotropy type is proved in the context of manifolds
  with corners. This procedure is shown to capture the simultaneous
  resolution of all isotropy types in a `resolution structure' consisting
  of equivariant iterated fibrations of the boundary faces. This structure
  projects to give a similar resolution structure for the quotient. In
  particular these results apply to give a canonical resolution of the
  radial compactification, to a ball, of any finite dimensional
  representation of a compact Lie group; such resolutions of the normal
  action of the isotropy groups appear in the boundary fibers in the
  general case.
\end{abstract}

\maketitle

\tableofcontents

\paperintro
\section*{Introduction} \label{Intro}

Borel showed that if the isotropy groups of a smooth action by a compact
Lie group, $G,$ on a compact manifold, $M,$ are all conjugate then the
orbit space, $G\backslash M,$ is smooth. Equivariant objects on $M,$ for
such an action, can then be understood directly as objects on the
quotient. In the case of a free action, which is to say a principal
$G$-bundle, Borel showed that the equivariant cohomology of $M$ is then
naturally isomorphic to the cohomology of $G \backslash M.$ In a companion
paper, \cite{Albin-Melrose:eqcores}, this is extended to the unique
isotropy case to show that the equivariant cohomology of $M$ reduces to the
cohomology of $G \backslash M$ with coefficients in a flat bundle (the
Borel bundle). In this paper we show how, by resolution, a general smooth
compact group action on a compact manifold is related to an action with
unique isotropy type on a resolution, canonically associated to the given
action, of the manifold to a compact manifold with corners.

The resolution of a smooth Lie group action is discussed by Duistermaat and
Kolk \cite{Duistermaat-Kolk} (which we follow quite closely), by Kawakubo
\cite{Kawakubo} and by Wasserman \cite{Wasserman} but goes back at least as
far as J\"anich \cite{Janich}, Hsiang \cite{Hsiang}, and Davis
\cite{Davis}. See also the discussion by Br\"uning, Kamber and Richardson
\cite{Richardson} which appeared after the present work was complete. In
these approaches there are either residual finite group actions,
particularly reflections, as a consequence of the use of real projective
blow up or else the manifold is repeatedly doubled. Using radial blow up,
and hence working in the category of manifolds with corners, such problems
do not arise.

For a general group action, $M$ splits into various isotropy types
\begin{equation*}
M^{[K]} = \{ \zeta \in M: G_{\zeta} \text{ is conjugate to } K \},\
	G_\zeta = \{ g \in G: g \zeta = \zeta \},\ \zeta \in M.
\end{equation*}
These are smooth manifolds but not necessarily closed and the orbit space
is then in general singular. We show below that each $M^{[K]}$ has a
natural compactification to a manifold with corners, $Y_{[K]},$ the
boundary hypersurfaces of which carry equivariant fibrations with bases the
compactifications of the isotropy types contained in the closure of
$M^{[K]}$ and so corresponding to larger isotropy groups. Each
fiber of these fibrations is the canonical resolution of the normal action
of the larger isotropy group. These fibrations collectively give what we
term a \emph{resolution structure}, $\{(Y_I,\phi_I);I\in\cI\},$ the index
set being the collection of conjugacy classes of isotropy groups, i.e. of
isotropy types, of the action. If $M$ is connected there is always a
minimal `open' isotropy type $\mu\in\cI,$ for which the corresponding
manifold, $Y_{\mu}=Y(M),$ (possibly not connected) gives a resolution of
the action on $M.$ That is, there is a smooth $G$-action on $Y(M)$ with
unique isotropy type and a smooth $G$-equivariant map
\begin{equation}
\beta :Y(M)\longrightarrow M
\label{Eqcores.220}\end{equation}
which is a diffeomorphism of the interior of $Y(M)$ to the minimal isotropy
type. Here, $\beta$ is the iterated blow-down map for the
resolution. There is a $G$-invariant partition of the boundary
hypersurfaces of $Y(M)$ into non-self-intersecting collections $H_I,$
labelled by the non-minimal isotropy types $I\in\cI\setminus\{\mu\},$ and
carrying $G$-equivariant fibrations
\begin{equation}
\phi_I:H_I\longrightarrow Y_I.
\label{Eqcores.222}\end{equation}
Here $Y_I$ resolves the space $M_I,$ the closure of the corresponding
isotropy type $M^I,$
\begin{equation}
\beta _I:Y_I\longrightarrow M_I,\ \beta \big|_{H_I}=\beta_I\circ\phi_I.
\label{Eqcores.221}\end{equation}

Thus the inclusion relation between the $M_I$ corresponding to the
stratification of $M$ by isotropy types, is `resolved' into the
intersection relation between the $H_I.$ The resolution structure for $M,$ thought
of as the partition of the boundary hypersurfaces with each collection
carrying a fibration, naturally induces a resolution structure for
each $Y_I.$ Since the fibrations are equivariant the quotients $Z_I$ of the
$Y_I$ by the group action induce a similar resolution structure on the
quotient $Z(M)$ of $Y(M)$ which resolves the quotient, the orbit space,
$G\backslash M.$

As noted above, in a companion paper \cite{Albin-Melrose:eqcores}, various
cohomological consequences of this construction are derived. The `lifts' of
both the equivariant cohomology and equivariant K-theory of a manifold with
a group action to its resolution structure are described. These lifted
descriptions then project to corresponding realizations of these theories
on the resolution structure for the quotient. As a consequence of the forms
of these resolved and projected theories a `delocalized' equivariant
cohomology is defined, and shown to reduce to the cohomology of Baum,
Brylinski and MacPherson in the Abelian case in \cite{MR791098}. The
equivariant Chern character is then obtained from the usual Chern
character by twisting with flat coefficients and establishes an isomorphism 
between equivariant K-theory with complex coefficients and delocalized 
equivariant cohomology. Applications to equivariant index theory will be 
described in \cite{Albin-Melrose:eqindres}.

For the convenience of the reader a limited amount of background
information on manifolds with corners and blow up is included in the first
two sections. The abstract notion of a resolution structure on a manifold
with corners is discussed in \S\ref{Ifs} and the basic properties of
G-actions on manifolds with corners are described in \S\ref{Gacts}. The
standard results on tubes and collars are extended to this case in
\S\ref{Collar}. In \S\ref{sec:BounRes} it is shown that for a general
action the induced action on the set of boundary hypersurfaces can be
appropriately resolved. The canonical resolution itself is then presented in
\S\ref{sec:Resolution}, including some simple examples, and the induced
resolution of the orbit space is considered in \S\ref{Quot}. Finally 
\S\ref{sec:Maps} describes the resolution of an equivariant embedding and the
`relative' resolution of the total space of an equivariant fibration.

The authors are grateful to Eckhard Meinrenken for very helpful comments on
the structure of group actions, and to an anonymous referee for remarks improving
the exposition.

\paperbody
\section{Manifolds with corners}\label{mwc}

By a \emph{manifold with corners,} $M,$ we shall mean a topological manifold
with boundary with a covering by coordinate charts 
\begin{equation}
M=\bigcup_jU_j,\ F_j:U_j\longrightarrow U_j'\subset\bbR^{m,\ell} =
[0,\infty)^\ell \times \bbR^{m-\ell}, 
\label{Eqcores.1}\end{equation}
where the $U_j$ and $U'_j$ are (relatively) open, the $F_j$ are
homeomorphisms and the transition maps 
\begin{equation}
F_{ij}:F_i(U_i\cap U_j)\longrightarrow F_j(U_i\cap U_j),\
U_i\cap U_j\not=\emptyset
\label{Eqcores.2}\end{equation}
are required to be smooth in the sense that all derivatives are bounded on
compact subsets; an additional condition is imposed below. The ring of
smooth functions $\CI(M)\subset\CO(M)$ is fixed by requiring
$(F_j^{-1})^*(u\big|_{U_j})$ to be smooth on $U_j',$ in the sense that it
is the restriction to $U_j'$ of a smooth function on an open subset of
$\bbR^m.$ 

The part of the boundary of smooth codimension one, which is the union of
the inverse images under the $F_i$ of the corresponding parts of the
boundary of the $\bbR^{m,\ell},$ is dense in the boundary and the closure
of each of its components is a \emph{boundary hypersurface} of $M.$ More
generally we shall call a finite union of non-intersecting boundary
hypersurfaces a \emph{collective boundary hypersurface}.  We shall insist, as part
of the definition of a manifold with corners, that these boundary
hypersurfaces each be \emph{embedded,} meaning near each point of each of
these closed sets, the set itself is given by the vanishing of a local
smooth defining function $x$ which is otherwise positive and has
non-vanishing differential at the point. In the absence of this condition
$M$ is a \emph{tied manifold.} It follows that each collective boundary
hypersurface, $H,$ of a manifold with corners is globally the zero set of a
smooth, otherwise positive, \emph{boundary defining function}
$\rho_H\in\CI(M)$ with differential non-zero on $H;$ conversely $H$
determines $\rho _H$ up to a positive smooth multiple. The set of connected
boundary hypersurfaces is denoted $\cM_1(M)$ and the \emph{boundary faces}
of $M$ are the \emph{components} of the intersections of elements of
$\cM_1(M).$ We denote by $\cM_k(M)$ the set of boundary faces of
codimension $k.$ Thus if $F\in\cM_k(M)$ and $F'\in\cM_{k'}(M)$ then $F\cap
F'$ can be identified with the union over the elements of a subset
(possibly empty of course) which we may denote $F\cap
F'\subset\cM_{k+k'}(M).$ Once again it is convenient to call a subset of
$\cM_k(M)$ with non-intersecting elements a collective boundary face, and then 
the collection of intersections of the elements of two collective boundary faces
\emph{is} a collective boundary face.

\begin{figure}[htpb]
	\centering
\includegraphics{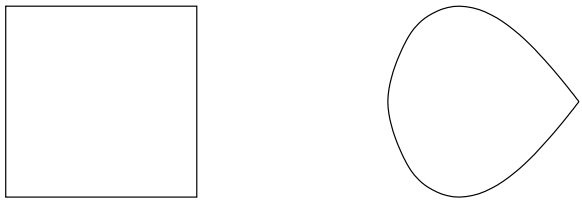}
\caption{The square is a manifold with corners. The teardrop is only a tied
  manifold since its boundary hypersurface intersects itself.}
\end{figure}

By a manifold from now on we shall mean a manifold with corners, so the
qualifier will be omitted except where emphasis seems appropriate. The
traditional object will be called a boundaryless manifold.

As a consequence of the assumption that the boundary hypersurfaces are
embedded, each boundary face of $M$ is itself a manifold with corners (for
a tied manifold the boundary hypersurfaces are more general objects, namely
\emph{articulated manifolds} which have boundary faces identified). At each
point of a manifold with corners there are, by definition, \emph{local
  product coordinates} $x_i\ge0,$ $y_j$ where $1\le i\le k$ and $1\le j\le
m-k$ (and either $k$ or $m-k$ can be zero) and the $x_i$ define the
boundary hypersurfaces through the point. Unless otherwise stated, by local
coordinates we mean local product coordinates in this sense. The local
product structure near the boundary can be globalized:-

\begin{definition}\label{Eqcores.123A} On a compact manifold with corners,
  $M,$ a \emph{boundary product structure} consists of a choice $\rho
  _H\in\CI(M)$ for each $H\in\cM_1(M),$ of a defining function for the
  each of the boundary hypersurfaces, an open neighborhood $U_H\subset M$
  of each $H\in\cM_1(M)$ and a smooth vector field $V_H$ defined in each $U_H$
  such that
\begin{equation}
\begin{gathered}
V_H\rho_K=\begin{cases}1&\Min U_H\Mif K=H\\0&\Min U_H\cap U_K\Mif K\not=H,
\end{cases}\\
[V_H,V_K]=0\Min U_H\cap U_K\ \forall\ H,K\in\cM_1(M).
\end{gathered}
\label{Eqcores.103}\end{equation}
\end{definition}
\noindent Integration of each $V_H$ from $H$ gives a product decomposition
of a neighborhood of $H$ as $[0,\epsilon _H]\times H,$ $\epsilon _H>0$ in
which $V_H$ is differentiation in the parameter space on which $\rho_H$
induces the coordinate. Shrinking $U_H$ allows it to be identified with
such a neighborhood without changing the other properties
\eqref{Eqcores.103}. Scaling $\rho _H$ and $V_H$ allows the parameter range
to be taken to be $[0,1]$ for each $H.$

\begin{proposition}\label{Eqcores.102} Every compact manifold has a
  boundary product structure.
\end{proposition}

\begin{proof} The construction of the neighborhoods $U_H$ and normal vector
  fields $V_H$ will be carried out inductively. For the inductive step it
  is convenient to consider a strengthened hypothesis. Note first that the
  data in \eqref{Eqcores.103} induces corresponding data on each boundary
  face $F$ of $M$ -- where the hypersurfaces containing $F$ are dropped,
  and for the remaining hypersurfaces the neighborhoods are intersected
  with $F$ and the vector fields are restricted to $F$ -- to which they are
  necessarily tangent. It may be necessary to subdivide the neighborhoods
  if the intersection $F\cap H$ has more than one component. In particular
  this gives data as in \eqref{Eqcores.103} but with $M$ replaced by $F.$
  So such data, with $M$ replaced by one of its hypersurfaces, induces data
  on all boundary faces of that hypersurface. Data as in
  \eqref{Eqcores.103} on a collection of boundary hypersurfaces of a
  manifold $M,$ with the defining functions $\rho_H$ fixed, is said to be
  consistent if all restrictions to a given boundary face of $M$ are the
  same.

Now, let $\cB\subset\cM_1(M)$ be a collection of boundary hypersurfaces
of a manifold $M,$ on which boundary defining functions $\rho _H$ have been
chosen for each $H\in\cM_1(M),$ and suppose that neighborhoods $U_K$ and
vector fields $V_K$ have been found satisfying \eqref{Eqcores.103} for all
$K\in\cB.$ If $H\in\cM_1(M)\setminus\cB$ then we claim that there is a
choice of $V_H$ and $U_H$ such that \eqref{Eqcores.103} holds for all
boundary hypersurfaces in $\cB\cup\{H\},$ with the neighborhoods possibly
shrunk. To see this we again proceed inductively, by seeking $V_H$ only on the
elements of a subset $\cB'\subset\cB$ but consistent on all common boundary
faces. The subset $\cB'$ can always be increased, since the addition of
another element of $\cB\setminus\cB'$ to $\cB'$ requires the same inductive
step but in lower overall dimension, which we can assume already
proved. Thus we may assume that $V_H$ has been constructed consistently on
all elements of $\cB.$ Using the vector fields $V_K,$ each of which is
defined in the neighborhood $U_K$ of $K,$ $V_H$ can be extended, locally
uniquely, from the neighborhood of $K\cap H$ in $K$ on which it is defined
to a neighborhood of $K\cap H$ in $M$ by demanding 
\begin{equation}
\cL_{V_K}V_H=[V_K,V_H]=0.
\label{Eqcores.104}\end{equation}
The commutation condition and other identities follow from this and the
fact that they hold on $K.$ Moreover, the fact that the $V_K$ commute in
the intersections of the $U_K$ means that these extensions of $V_H$ are
consistent for different $K$ on their common domains. In this way $V_H$
satisfying all conditions in \eqref{Eqcores.103} has been constructed in a
neighborhood of the part of the boundary of $H$ in $M$ corresponding to
$\cB.$ In the complement of this part of the boundary one can certainly
choose $V_H$ to satisfy $V_H\rho_H=1$ and combining these two choices using
a partition of unity (with two elements) gives the desired additional
vector field $V_H$ once the various neighborhoods $U_K$ are shrunk.

Thus, after a finite number of steps the commuting normal vector fields
$V_K$ are constructed near each boundary hypersurface.
\end{proof}

Note that this result is equally true if in the definition the set of
boundary hypersurfaces is replaced with any partition into collective
boundary hypersurfaces, however it is crucial that the different
hypersurfaces in each collection do not intersect.

The existence of such normal neighborhoods of the boundary hypersurfaces
ensures the existence of `product-type' metrics. That is, one can choose a
metric $g$ globally on $M$ which near each boundary hypersurface $H$ is of
the form $d\rho _H^2+\phi_H^*h_H$ where $\phi_H:U_H\longrightarrow H$ is
the projection along the integral curves of $V_H$ and $h_H$ is a metric,
inductively of the same product-type, on $H.$ Thus near a boundary face
$F\in\cM_k(M),$ which is defined by $\rho_{H_i},$ $i=1,\dots,k,$ the metric
takes the form
\begin{equation}
g=\sum\limits_{i=1}^kd\rho_{H_i}^2+\phi_F^*h_F
\label{Eqcores.106}\end{equation}
where $\phi_F$ is the local projection onto $F$ with leaves the integral
surfaces of the $k$ commuting vector fields $V_{H_i}.$ In particular

\begin{corollary}\label{Eqcores.105} On any manifold with corners there
  exists a metric $g,$ smooth and non-degenerate up to all boundary faces,
  for which the boundary faces are each totally geodesic.
\end{corollary}

A diffeomorphism of a manifold sends boundary faces to boundary faces --
which is to say there is an induced action on $\cM_1(M).$

\begin{definition}\label{Eqcores.107} A diffeomorphism $F$ of a manifold
  $M$ is said to be \emph{boundary intersection free} if for each
  $H\in\cM_1(M)$ either $F(H)=H$ or $F(H)\cap H=\emptyset.$ More generally
  a collection $\cG$ of diffeomorphisms is said to be boundary intersection
  free if $\cM_1(M)$ can be partitioned into collective boundary
  hypersurfaces $B_i\subset\cM_1(M),$ so the elements of each $B_i$ are
  disjoint, such that the induced action of each $F\in\cG$ preserves the
  partition, i.e. maps each $B_i$ to itself. 
\end{definition}

A manifold with corners, $M,$ can always be realized as an embedded
submanifold of a boundaryless manifold. As shown in \cite{daomwc}, if
$\cF\subset\cM_1(M)$ is any disjoint collection of boundary hypersurfaces
then the `double' of $M$ across $\cF,$ meaning $2_{\cF}M=M\sqcup M/\cup\cF$
can be given (not however naturally) the structure of a smooth manifold
with corners. If $\{ \cF_1, \ldots \cF_\ell \}$ is a partition of the boundary of $M$ into 
disjoint collections, then it induces a partition $\{ \wt \cF_2, \ldots \wt \cF_\ell \}$ of the boundary of $2_{\cF_1}M$ with one less element.
After a finite number of steps, the iteratively doubled manifold is
boundaryless and $M$ may be identified with the image of one of the
summands (see Theorem \ref{DoublingThm}).

\begin{figure}[htpb]
\centering
\includegraphics{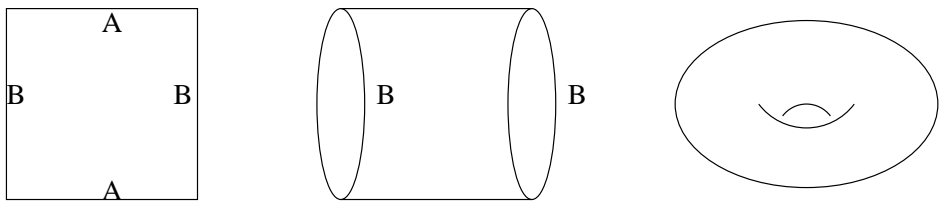}
\caption{After doubling the boundaries marked $A$ and then doubling the boundaries marked $B$ we end up with a torus.}
\end{figure}

\section{Blow up}\label{Blowup}

A subset $X\subset M$ of a manifold (with corners) is said to be a
\emph{p-submanifold} if at each point of $X$ there are local (product)
coordinates for $M$ such that $X\cap U,$ where $U$ is the coordinate
neighborhood, is the common zero set of a subset of the coordinates.  An
\emph{interior p-submanifold} is a p-submanifold no component of which is
contained in the boundary of $M.$ 

\begin{figure}[htpb]
	\centering
	\includegraphics{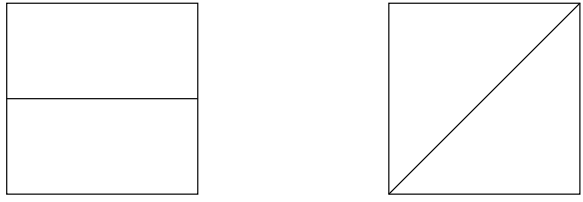}
	\caption{A horizontal line is an interior p-submanifold of the
          square. The diagonal in a product of manifolds with boundary is
          not a p-submanifold.}
\end{figure}

A p-submanifold of a manifold is itself a manifold with corners, and the 
collar neighborhood theorem holds in this context. Thus the normal bundle
to $X$ in $M$ has (for a boundary p-submanifold) a well-defined 
inward-pointing subset, forming a submanifold with corners $N^+X\subset NX$
(defined by the non-negativity of all $d\rho_H$ which vanish on the 
submanifold near the point) and, as in the boundaryless case, the 
exponential map, but here for a product-type metric, gives a diffeomorphism 
of a neighborhood of the zero section with a neighborhood of $X:$
\begin{equation}
T:N^+X\supset U' \longrightarrow U \subset M.
\label{Eqcores.117}\end{equation}
The radial vector field on $N^+X$ induces a vector field $R$ near $X$ which
is tangent to all boundary faces.

\begin{proposition}\label{Eqcores.118} If $X$ is a closed p-submanifold in a
  compact manifold then the boundary product structure in
  Proposition~\ref{Eqcores.102}, for any choice of boundary defining
  functions, can be chosen so that $V_H$ is tangent to $X$ unless $X$ is
  contained in $H.$
\end{proposition}

\begin{proof} The condition that the $V_H$ be tangent to $X$ can be carried
  along in the inductive proof in Proposition~\ref{Eqcores.102}, starting
  from the smallest boundary face which meets $X.$
\end{proof}

If $X\subset M$ is a closed p-submanifold then the radial blow-up of $M$
along $X$ is a well-defined manifold with corners $[M;X]$ obtained from $M$
by replacing $X$ by the inward-pointing part of its spherical normal
bundle. It comes equipped with the blow-down map
 \begin{equation}
[M;X] = S^+X \sqcup (M \setminus X),\ \beta :[M;X]\longrightarrow M.
\label{Blow-up}
\end{equation}
The preimage of $X,$ $S^+X,$ is the `front face' of the blow up, denoted
$\ff([M;X]).$ The natural smooth structure on $[M;X],$ with respect to
which $\beta$ is smooth, is characterized by the additional condition that
a radial vector field $R$ for $X,$ as described above, lifts under $\beta$
(i\@.e\@.~is $\beta$-related) to $\rho_{\ff}X_{\ff}$ for a defining
function $\rho _{\ff}$ and normal vector field $X_{\ff}$ for the new
boundary introduced by the blow up.

\begin{figure}[htpb]
\centering
\includegraphics{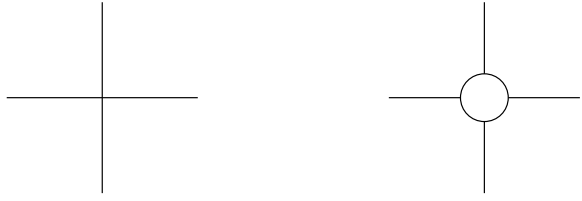}
\caption{Blowing up the origin in $\bbR^2$ results in the manifold with boundary $[\bbR^2; \{0\}] = \bbS^1 \times \bbR^+.$ Polar coordinates around the origin in $\bbR^2$ yield local coordinates near the front face in $[\bbR^2; \{0\}].$ }
\end{figure}

Except in the trivial cases that $X=M$ or $X\in\cM_1(M)$ the front face is
a `new' boundary hypersurface of $[M;X]$ and the preimages of the boundary
hypersurfaces of $M$ are unions of the other boundary hypersurfaces of
$[M;X];$ namely the lift of $H$ is naturally $[H;X\cap H].$ So, in the
non-trivial cases and unless $X$ separates some boundary hypersurface into
two components, there is a natural identification
\begin{equation}
\cM_1([M;X])=\cM_1(M)\sqcup\{\ff([M;X])\}
\label{Eqcores.3}\end{equation}
which corresponds to each boundary hypersurface of $M$ having a unique
`lift' to $[M;X],$ as the boundary hypersurface which is the closure of the
preimage of its complement with respect to $X.$ In local coordinates,
blowing-up $X$ corresponds to introducing polar coordinates around $X$ in
$M.$

\begin{lemma}\label{Eqcores.119} If $X$ is a closed interior
  p-submanifold and $M$ is equipped with a boundary product structure in
  the sense of Proposition~\ref{Eqcores.102} the normal vector fields of
  which are tangent to $X$ then the radial vector field for $X$ induced by
  the exponential map of an 
  associated product-type metric commutes with $V_H$ near any $H\in\cM_1(M)$
  which intersects $X$ and on lifting to $[M;X],$ $R=\rho_{\ff}X_{\ff}$ 
  where $\rho _{\ff}$ and $X _{\ff},$ together with the lifts of the $\rho
  _H$ and $V_H$ give a boundary product structure on $[M;X].$
\end{lemma}

\begin{proof} After blow up of $X$ the radial vector field lifts to be of
  the form $a\rho _{\ff}V_{\ff}$ for any normal vector field and defining
  function for the front face, with $a>0.$ The other product data lifts to
product data for all the non-front faces of $[M;X]$ and this lifted data
satisfies $[R,V_H]=0$ near $\ff.$ Thus it is only necessary to show, using
an inductive argument as above, that one can choose $\rho_{\ff}$ to satisfy
$V_H\rho _{\ff}=0$ and $R\rho_{\ff}=\rho_{\ff}$ in appropriate sets to
conclude that $R=\rho _{\ff}V_{\ff}$ as desired.
\end{proof}

\section{Resolution structures}\label{Ifs}

A fibration is a surjective smooth map $\Phi : H\longrightarrow Y$ between
manifolds with the property that for each component of $Y$ there is a
manifold $Z$ such that each point $p$ in that component has a neighborhood
$U$ for which there is a diffeomorphism giving a commutative diagram with
the projection onto $U:$
\begin{equation}
\xymatrix{
\Phi^{-1}(U)\ar[rr]^{F_U}\ar[dr]_{\Phi}&&Z\times U\ar[dl]^{\pi_U}\\
&U.
}
\label{Eqcores.120}\end{equation}
The pair $(U, F_U)$ is a local trivialization of $\Phi.$
Set $\codim(\phi)=\dim Z,$ which will be assumed to be the same for all
components of $Y.$ The image of a boundary face under a fibration must
always be a boundary face (including the possibility of a component of $Y).$ 

\begin{lemma} \label{BlowUpLemmaFib}
Suppose $\Phi: H \longrightarrow Y$ is a fibration with typical fiber $Z.$
\begin{enumerate}
\item[i)] If $S \subseteq H$ is a closed p-submanifold
  transverse to the fibers of $\Phi$, then the composition of $\Phi$ with
  the blow-down map $\beta:[H;S] \longrightarrow H$ is a fibration.
\item[ii)] If  $T \subseteq Y$ is a closed interior
  p-submanifold, then $\Phi$ lifts from $H \setminus \Phi^{-1}(T)$ to a fibration
  $\beta^{\#}\Phi:[H; \Phi^{-1}(T)] \longrightarrow [Y;T].$
\end{enumerate}
\end{lemma}

\begin{remark}\label{resolution.1}
In the situation of ii), one may consider instead the pull-back
fibration 
\begin{equation*}
\xymatrix{ \beta_Y^*H \ar[d] \ar[r] & H \ar[d]^{\Phi} \\
[Y;T] \ar[r]^{\beta_Y} & Y }
\end{equation*}
where $\beta_Y^*H=\{(\zeta,\xi)\in H\times[Y;T]:\Phi(\zeta)=\beta_Y(\xi) \}.$
The natural map $[H; \Phi^{-1}(T)] \ni \alpha \mapsto (\beta_H(\alpha),
\wt\Phi(\alpha) ) \in \beta_Y^*H$ is a diffeomorphism, showing that these
fibrations coincide.
\end{remark}

\begin{proof}
i) Transversality ensures that $\Phi(S) = Y$ and so $\Phi\rest{S}$
is itself a fibration, say with typical fiber $Z_S.$ 
If $(U, F_U)$ is a local trivialization of $\Phi$
then since
\begin{equation*}
	[U \times Z; U \times Z_S] =
	U \times [Z; Z_S],
\end{equation*}
the diffeomorphism $F_U$ induces a diagram
\begin{equation*}
	\xymatrix{
	(\beta^*\Phi)^{-1}(U)\ar[rr] \ar[dr]_{\beta^*\Phi} 
	& & U \times [Z; Z_S]  \ar[dl]^{\pi_U}\\
	& U}
\end{equation*}
which shows that $\beta^*\Phi:[H;S] \longrightarrow H \longrightarrow Y$ is
a fibration. 

\noindent ii)
Let $(U,F_U)$ be a local trivialization of $\Phi$ and $T_U = T \cap U.$
The diffeomorphism $F_U$ identifies $\Phi^{-1}(U)$ with $Z \times U$ and $\Phi^{-1}(T_U)$ with $Z \times T_U$
and so lifts to a diffeomorphism $\wt F_U$ of $(\beta^{\#}\Phi)^{-1}([U; T_U])$ with $Z \times [U ; T_U] = [Z\times U; Z\times T_U].$
Thus $([U; T_U], \wt F_U)$ is a local trivialization for $\beta^{\#}\Phi,$
\begin{equation*}
	\xymatrix{
	(\beta^{\#}\Phi)^{-1}([U;T_U])\ar[rr]^-{\wt F_U }\ar[dr]_{\beta^{\#}\Phi} 
	& & Z \times [U; T_U]  \ar[dl]^{\pi_U}\\
	& [U; T_U] }
\end{equation*}
which shows that  $\beta^{\#}\Phi:[H; \Phi^{-1}(T)] \longrightarrow [Y;T]$ is a fibration.
\end{proof}

The restriction of the blow-down map to the boundary hypersurface
introduced by the blow up of a p-submanifold is a fibration, just the
bundle projection for the (inward-pointing part of) the normal sphere
bundle. In general repeated blow up will destroy the fibration property of
this map. However in the resolution of a $G$-action the fibration condition
persists. We put this into a slightly abstract setting as follows.

\begin{definition}\label{Eqcores.121} A resolution structure on a
manifold $M$ is a partition of $\cM_1(M)$ into collective boundary
hypersurfaces, each with a fibration, $\phi_H:H\longrightarrow Y_H$ with 
the consistency properties that if $H_i\in\cM_1(M),$ $i=1,2,$ and 
$H_1\cap H_2\not=\emptyset$ then $\codim(\phi_{H_1})\not=\codim(\phi_{H_2})$ 
and
\begin{equation}
\begin{gathered}
\codim(\phi_{H_1})<\codim(\phi_{H_2})\Longrightarrow\\
\phi_{H_1}(H_1\cap H_2)\in\cM_1(Y_{H_1}),\ \phi_{H_2}(H_1\cap H_2)=Y_{H_2}
\Mand \exists\text{ a fibration}\\
\phi_{H_1H_2}:\phi_{H_1}(H_1\cap H_2)\longrightarrow Y_{H_2}
\text{ giving a commutative diagram:}\\
\xymatrix{
H_1\cap H_2\ar[dr]^{\phi_{H_2}}\ar[rr]^{\phi_{H_1}}&&
\phi_{H_1}(H_1\cap H_2)\ar[dl]_{\phi_{H_1H_2}}\\
&Y_{H_2}.
}
\end{gathered}
\label{Eqcores.122}\end{equation}
\end{definition}

\begin{lemma}\label{Eqcores.301} A resolution structure induces 
resolution structures on each of the manifolds $Y_H.$
\end{lemma}

\begin{proof} Each boundary hypersurface $F$ of $Y_H$ is necessarily the
  image under $\phi_H$ of a unique boundary hypersurface of $H,$ therefore
  consisting of a component of some intersection $H\cap K$ for
  $K\in\cM_1(M).$ The condition \eqref{Eqcores.122} ensures that
  $\codim(\phi_H)<\codim(\phi_K)$ and gives the fibration
  $\phi_{HK}:F\longrightarrow Y_K.$ Thus for $Y_H$ the bases of the
  fibrations of its boundary hypersurfaces are all the $Y_K$'s with the
  property that $H\cap K\not=\emptyset$ and $\codim(\phi_H)<\codim(\phi_K)$
  with the fibrations being the appropriate maps $\phi_*$ from
  \eqref{Eqcores.122}.

Similarly the compatibility maps for the boundary fibration of $Y_H$ follow
by the analysis of the intersection of three boundary hypersurfaces $H,$
$K$ and $J$ where $\codim(\phi_H)<\codim(\phi_K)<\codim(\phi_J).$ Any two
intersecting boundary hypersurfaces of $Y_H$ must arise in this way, as
$\phi_H(H\cap K)$ and $\phi_H(H\cap J)$ and the compatibility map for them
is $\phi_{JK}.$
\end{proof}

If $M$ carries a resolution structure then Lemma \ref{BlowUpLemmaFib}
shows that appropriately placed submanifolds can be blown up and the
resolution structure can be lifted. Specifically we say that a manifold 
$T$ is {\em transverse to the resolution structure} if either:

\begin{enumerate}
\item[i)] $T$ is an interior p-submanifold of $M,$ with $\dim T < \dim M,$ that 
is transverse to the fibers of $\phi_H$ for all $H \in \cM_1(M),$ or
\item[ii)] $T$ is an interior p-submanifold of $Y_L,$ for some $L \in \cM_1(M),$
with $\dim T < \dim Y_L,$ that is transverse to the fibers of $\phi_N$ for
all $N \in \cM_1(Y_L).$
\end{enumerate}
Let $\wt T \subseteq M$ be equal to $T$ in the first case and
$\phi_L^{-1}(T)$ in the second, then we have the following result.

\begin{proposition}\label{Eqcores.126} 
If $M$ carries a resolution structure and $T$ is a manifold transverse to 
it, then $[M; \wt T]$ carries a resolution structure.
In case ii) above, where $T \subseteq Y_L,$ the resolution structure on 
$[M; \phi_L^{-1}(T)]$ is obtained by blowing-up the lift of $T$ to every 
$Y_K$ that fibers over $Y_L.$
In both cases, at each boundary face of the new resolution structure
the boundary fibration is either the pull-back of the previous one 
along the blow-down map or the blow-down map itself.
\end{proposition}
\noindent Recall that submanifolds which do not intersect are included in
the notion of transversal intersection.

\begin{proof} Consider the two cases in the definition of transverse
  submanifold separately. (For clarity, we assume throughout the proof
  that the collective boundary hypersurfaces in Definition 
  \ref{Eqcores.121} consist of a single boundary hypersurface.)

\noindent
Case i). Let $\beta_T:[M;T] \longrightarrow M$ be the blow-down map. 
A boundary face of $[M;T]$ is either the lift of a boundary face 
$H \in \cM_1(M),$ in which case $\beta_T^*\phi_H$ is a fibration by Lemma 
\ref{BlowUpLemmaFib} i), or it is the front face of the blow-up, in 
which case it carries the fibration $\beta_T\big|_{\ff}.$ Thus we only 
need to check the compatibility conditions.

The compatibility maps for the fibrations of the hypersurfaces of $M$
clearly lift to give compatibility maps for the lifts. Thus it is only
necessary to check compatibility between the fibrations on these lifted 
boundary hypersurfaces of $[M;T]$ and that of the front face. So, let $H$ 
be a hypersurface of $M$ that intersects $T.$ In terms of the notation
above, the codimension of $\beta_T^*\phi_H$ is the equal to $\dim Z_H$ 
while the codimension of $\phi_{\ff}$ is equal to $\dim Z_H - 
\dim Z_{H \cap T}.$ The diagram \eqref{Eqcores.122} in this case is
\begin{equation*}
\xymatrix{
\ff \cap [H;H\cap T] \ar[dr]^{\wt \phi_{H}}\ar[rr]^{\beta_T}&&
H \cap T \ar[dl]_{\phi_H}\\
&\phi_H(H\cap T) = Y_H.
}
\end{equation*}
and so the requirements of Definition \ref{Eqcores.121} are met.

\noindent Case ii).
First note that the inverse image of a p-submanifold under a fibration is
again a p-submanifold since this is a local property and locally a fibration 
is a projection. We denote by $\beta_T:[M; \phi_L^{-1}(T)] \longrightarrow
M$ the blow-down map and make use of the notation in \eqref{Eqcores.122}.

From the front face the map
\begin{equation*}
	\ff([M; \phi_L^{-1}(T)]) \xrightarrow{\text{ }\beta_T\text{ }} 
	\phi_L^{-1}(T) \xrightarrow{\text{ }\phi_L\text{ }} T
\end{equation*}
is the composition of fibrations and so is itself a fibration.

Consider the lift of a boundary face $H \in \cM_1(M)$ to a boundary face 
of $[M; \wt T].$ If $H \cap \phi_L^{-1}(T)$ is empty then 
$\beta_T^*\phi_H$ fibers over $Y_H$ and the compatibility conditions are 
immediate. If $H \cap \phi_L^{-1}(T)$ is not empty and $\codim(\phi_L) <
\codim (\phi_H)$ then, by Lemma \ref{BlowUpLemmaFib}, $\beta_T^*\phi_H$
fibers over $Y_H$ and the arrows in the commutative diagrams 
\begin{equation*}
	\xymatrix{ [H \cap L; H \cap \phi_L^{-1}(T)] 
	\ar[dr]^{\beta_T^*\phi_H} \ar[rr]^{\beta_T^{\#}\phi_L} 
	&& [\phi_L(H \cap L); \phi_L(H \cap L) \cap T] 
	\ar[dl]_{\beta^*\phi_{LH}} \\
	& Y_H & }
\end{equation*}
and
\begin{equation*}
	\xymatrix{ \ff([H; H \cap \phi_L^{-1}(T)]) 
	\ar[dr]^{\beta_T^*\phi_H} \ar[rr]^{\beta_T^{*}\phi_L} 
	&& \phi_L(H \cap L) \cap T \ar[dl]_{\phi_{LH}} \\
	& Y_H & }
\end{equation*}
are all fibrations. Here, surjectivity of $\phi_{LH}\rest{\phi_L(H \cap L)
  \cap T}$ follows from the transversality of $T$ to the fibers of
$\phi_{LH}.$ Since the  lift of $H$ meets the lift of $L$ in $[H \cap L; H
  \cap \phi_L^{-1}(T)]$ and meets the front face of $[M; \phi_L^{-1}(T)]$
in $\ff([H; \phi_L^{-1}(T) \cap H]),$ these diagrams also establish the
compatibility conditions for the lift of $H.$

Next if $H \cap \phi_L^{-1}(T)$ is not empty and $\codim(\phi_L) > \codim
(\phi_H),$  then Lemma \ref{BlowUpLemmaFib} guarantees that the map
$\beta_T^{\#}\phi_H$ is a fibration from the lift of $H$ to $[Y_H;
  \phi_{HL}^{-1}(T)]$ and that the arrows in the commutative diagrams
\begin{equation*}
	\xymatrix{ [H \cap L; H \cap \phi_L^{-1}(T)] 
	\ar[dr]^{\beta_T^{\#}\phi_H} \ar[rr]^{\beta_T^{\#}\phi_H} 
	&& [\phi_L(H \cap L); \phi_L(H \cap L) \cap T] 
	\ar[dl]_{\beta^{\#}\phi_{HL}} \\
	& [Y_L; T] & }
\end{equation*}
and
\begin{equation*}
	\xymatrix{ \ff([H; H \cap \phi_L^{-1}(T)]) 
	\ar[dr]^{\beta_T^*\phi_H} \ar[rr]^{\beta_T^{\#}\phi_L} 
	&& \ff([\phi_H(H \cap L); \phi_{HL}^{-1}(T)]) 
	\ar[dl]_{\beta^*\phi_{HL}} \\
	& T.}
\end{equation*}
are all fibrations.

Finally consider the lift of $L.$ The map $\beta^{\#}\phi_L:[L; \phi_L^{-1}(T)] 
\longrightarrow [Y_L; T]$ is a fibration by Lemma \ref{BlowUpLemmaFib} and
the discussion above shows that it is compatible with the lift of $H$ for
any $H \in \cM_1(M).$ The final compatibility between the lift
of $L$ and the front face is established by the commutative diagram
\begin{equation*}
\xymatrix{
\ff( [L; \phi_L^{-1}(T)] ) \ar[dr]^{\beta_T^* \phi_{L}}\ar[rr]^{\phi_{L}}&&
\ff( [Y_L;T] )  \ar[dl]_{\beta}\\
&T.
}
\end{equation*}

\end{proof}

\begin{definition}\label{Eqcores.123} If $M$ carries a resolution
  structure as in Definition~\ref{Eqcores.121} then a boundary product
  structure is said to be \emph{compatible} with the resolution structure if
  for each pair of intersecting boundary faces $H_1$ and $H_2$ with 
  $\codim(\phi_{H_1})<\codim(\phi_{H_2})$  
\begin{gather}
\rho_{H_2}\big|_{H_1}\in\phi_{H_1}^*\CI(Y_{H_1})\Mnear H_2,
\label{Eqcores.124}\\
V_{H_2}\big|_{H_1}\text{ is }\phi_{H_1}\text{-related to a vector field on }
Y_{H_1}\Mnear H_2
\Mand\\
V_{H_1}\big|_{H_2}\text{ is tangent to the fibers of }\phi_{H_2}.
\label{Eqcores.125}\end{gather}
\end{definition}

\begin{proposition}\label{PreserveIFS} For any resolution structure
on a compact manifold, $M,$ there is a compatible boundary product
structure.
\end{proposition}

\begin{proof} We follow the proof on Proposition \ref{Eqcores.102}. In
  particular, we will use the notion of consistent boundary data on a
  collection of boundary hypersurfaces.

First, choose boundary defining functions satisfying \eqref{Eqcores.124}.
Let $H \in \cM_1(M)$ and define $\cH\subset\cM_1(M)$ to consist of those
boundary hypersurfaces $K\in\cM_1(M)$ which intersect $H$ and satisfy
$\codim(\phi_K) < \codim (\phi_H).$ If $L \in \cH,$ we may assume
inductively that we have chosen $\rho_H\rest K$ for all boundary
hypersurfaces $K \in \cH$ with $\codim(\phi_K) < \codim (\phi_L),$ and then
choose an extension to $H \cap L$ as a lift of a boundary defining function
for the boundary face $\phi_{L}(H \cap L).$ This allows $\rho_H$ to be
defined on a neighborhood of $H \cap K$ in $K$ for all $K \in \cH;$
extending it to a boundary defining function of $H$ in $M$ fulfills the
requirements.

Next, suppose normal vector fields consistent with the resolution
structure and associated collar neighborhoods have been found for some
subset $\cB \subseteq \cM_1(M)$ with the property that $H \in
\cM_1(M)\setminus \cB$ and $K \in \cB$ implies that $H\cap K=\emptyset$ or
$\codim(\phi_H) < \codim(\phi_K).$ Let $H \in \cM_1(M)\setminus \cB$ be
such that $\phi_H$ has maximal codimension among the boundary hypersurfaces
of $M$ that are not in $\cB.$ We show that there is a choice of $V_H$ and
$U_H$ such that \eqref{Eqcores.103} and the conditions of Definition
\ref{Eqcores.123} hold for all boundary hypersurfaces in $\cB \cup \{H\}.$

As before an inductive argument allows us to find $V_H$ in a neighborhood
of all intersections $H \cap K$ with $K \in \cB$ with the property that
$V_H\rest{K}$ is tangent to the fibers of $\phi_K.$ Then $V_H$ can be
extended into $U_K$ using the vector fields $V_K$ by demanding that
\begin{equation*}
	\cL_{V_K}V_H = [V_K, V_H] = 0
\end{equation*}
thus determining $V_H$ locally uniquely in a neighborhood of $H \cap K$ in
$M$ for all $K \in \cB.$

If $K \in \cM_1(M)\setminus \cB$ intersects $H,$ then $Y_K$ is itself a
manifold with a resolution structure and $\phi_H(H \cap K)$ is one
of its boundary hypersurfaces. We can choose boundary product data on $Y_K$
-- since it has smaller dimension than $M$ we may assume that the
proposition has been proven for it. Under a fibration there is always a
smooth lift of vector fields, a connection, so $V_H$ on $\phi_H(H \cap K)$
may be lifted to a vector field $V_H$ on $H \cap K.$ In this way $V_H$ may
be chosen on the intersection of $H$ with any of its boundary faces. Then
$V_H$ may be extended into a neighborhood $U_H$ of $H$ in $M$ in such a way
that $V_H\rho_H =1.$  By construction the commutation
relations with all the previously constructed vector fields are satisfied
and $V_H$ is compatible with the resolution structure at all
boundary hypersurfaces in $\cB \cup \{ H \}.$ Thus the inductive step is
justified.
\end{proof}

Using Proposition~\ref{Eqcores.118} and Lemma~\ref{Eqcores.119} we see that
resolution structures and boundary product structures are preserved
when blowing up appropriately placed p-submanifolds.

\begin{proposition}\label{Eqcores.124A} 
If $M$ is a manifold with a resolution structure and 
$X$ is a manifold transverse to the resolution structure then $[M;\wt X]$ 
has a boundary product structure which is
  compatible with the resolution structure given by
  Proposition~\ref{Eqcores.126}, is such that the normal vector fields to
  boundary hypersurfaces other than the front face are $\beta$-related to a
  boundary product structure on $M$ and is such that $\rho _{\ff}V_{\ff}$
  is $\beta$-related to a radial vector field for $\wt X.$
\end{proposition}

\section{Group actions}\label{Gacts}

Let $G$ be a compact Lie group and $M$ a compact manifold (with
corners). An action of $G$ on $M$ is a smooth map $A:G\times
M\longrightarrow M$ such that $A(\Id,\zeta)=\zeta$ for all $\zeta\in M$ and
\begin{equation}
\xymatrix{
&G\times M\ar[dr]^{A}\\
G\times G\times M\ar[ur]^{\cdot\times\Id}\ar[dr]_{\Id\times A}&&M\\
&G\times M\ar[ur]_A}
\label{Eqcores.101}\end{equation}
commutes; here $\cdot$ denotes the product in the group. Equivalently this
is just the requirement that $A$ induces a group homomorphism from $G$ to
the diffeomorphism group of $M.$

We will usually denote $A(g,\zeta)$ as $g\cdot \zeta.$ Since each element
$g\in G$ acts as a diffeomorphism on $M,$ it induces a permutation of the
boundary hypersurfaces of $M.$  If $g$ is in the connected component of the
identity of $G,$ this is the trivial permutation.

Our convention is to assume, as part of the definition, that the action of
$G$ is boundary intersection free in the sense of
Definition~\ref{Eqcores.107}. That is, the set $\cM_1(M)$ of boundary
hypersurfaces can be partitioned into disjoint sets
\begin{equation}
\begin{gathered}
\cM_1(M)=B_1\sqcup B_2\sqcup\dots\sqcup B_l,\ H,H'\in B_i\Longrightarrow
H\cap H'=\emptyset, \\ \text{ and s\@.t\@. } g\cdot H\in B_i\Mif H\in B_i. 
\end{gathered}
\label{Eqcores.108}\end{equation}
The contrary case will be referred to as a $G$-action \emph{with boundary
  intersection} -- it is shown below in Proposition~\ref{BounRes} that by
resolution the boundary intersection can be removed. As justification for
our convention, note that the $G$-actions which arise from the resolution of
a $G$-action on a manifold without boundary are always boundary
intersection free.

\begin{figure}[htpb]
\label{fig:BdyFree}
\centering
\includegraphics{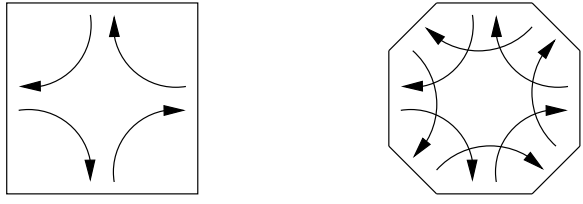}
\caption{The action on the square has boundary intersection, that on the octagon is boundary intersection free.}
\end{figure}

For a given $G$-action, the \emph{isotropy} (or \emph{stabilizer}) subgroup
of $G$ at $\zeta \in M$ is
\begin{equation}
G_{\zeta} = \{ g \in G; g\cdot \zeta = \zeta \}.  
\label{Eqcores.111}\end{equation}
It is a closed, and hence Lie, subgroup of $G.$

The action of $G$ on $M$ induces a pull-back action on $\CI(M).$ The
differential of this action at $\Id\in G$ induces the action of the Lie
algebra $\df g$ on $\CI(M)$ where $V \in\df g$ is represented by a vector
field $\alpha(V) \in \bV(M),$ the Lie algebra of smooth vector fields on $M$
tangent to all boundary faces, given by 
\begin{equation}\label{DefLX}
\alpha(V)f(\zeta) = \frac{d}{d t} f \lrpar{ e^{-tV} \zeta }\rest{t =0},
\Mforall f \in \CI(M).
\end{equation}
Since $[\alpha(V), \alpha(W)] = \alpha([V,W]),$ this is a map of Lie
algebras, $\alpha :\df g \longrightarrow \bV(M).$ The differential at
$\zeta\in M$ will be denoted
\begin{equation}
\alpha _\zeta :\df g\longrightarrow T_\zeta M.
\label{Eqcores.112}\end{equation}
The image always lies in $T_\zeta F$ where $F\in\cM_k(M)$ is the
smallest boundary face containing $\zeta.$
 
\begin{proposition}\label{Eqcores.109} For any compact group action on a compact
  manifold, satisfying \eqref{Eqcores.108}, the collective boundary
  hypersurfaces $B_i$ each have a collective defining function
  $\rho_i\in\CI(M)$ which is $G$-invariant and there is a corresponding
  $G$-invariant product structure near the boundary consisting of smooth
  $G$-invariant vector fields $V_i$ and neighborhoods $U_i$ of
  $\supp(B_i)=\cup\{H\in B_i\}$ for each $i$ such that  
\begin{equation}
V_i\rho_j=\begin{cases}1&\Min U_i\Mif i=j\\
0&\Min U_i\cap U_j\Mif i\not=j.
\end{cases}
\label{Eqcores.110}\end{equation}
Furthermore there is a $G$-invariant product-type metric on $M.$
\end{proposition}

\begin{proof} Any collective boundary hypersurface has
  a common defining function, given by any choice of defining function
  near each boundary hypersurface in the set extended to be strictly
  positive elsewhere.  If $\rho_i'$ is such a defining function for
  $\supp(B_i)$ then so is $g^*\rho_i$ for each $g\in G,$ since by
  assumption it permutes the elements of $B_i.$ Averaging over $G$ gives a
  $G$-invariant defining function. Similarly each of the vector fields
  $V_H$ in \eqref{Eqcores.103} is only restricted near $H$ so these can be
  combined to give collective normal vector fields $V_i$ which then have
  the properties in \eqref{Eqcores.110}. Since the commutation conditions
  are bilinear they cannot be directly arranged by averaging, but the
  normal vector fields can be constructed, and averaged, successively.

A product-type metric made up (iteratively) from this invariant data near
the boundary can similarly be averaged to an invariant product-type
metric. In fact the average of any metric for which the boundary faces are
all totally geodesic has the same property.
\end{proof}

One direct consequence of the existence of an invariant product structure
near the boundary is that, as noted above, a smooth group action on a
manifold with corners can be extended to a group action on a closed
manifold. This allows the consideration of the standard properties of group
actions to be extended trivially from the boundaryless case to the case
considered here of manifolds with corners. 

\begin{theorem} \label{DoublingThm} Suppose $M$ is a compact manifold with
  corners with a smooth action by a compact Lie group $G$ -- so assumed to
  satisfy \eqref{Eqcores.108} -- then if $M$ is doubled successively,
  as at the end of \S\ref{mwc}, across the elements of a partition into $l$
  $G$-invariant collective boundary hypersurfaces,
  to a manifold without boundary, $\double{M},$ then there is a smooth
  action of $\bbZ_2^l \times G$ on $\double{M}$ such that $M$ embeds
  $G$-equivariantly into $\double{M}$ as a fundamental domain for the
  $\bbZ_2^l$-action.
\end{theorem}

\begin{proof} (See \cite[Chapter 1]{daomwc} and \cite[\S
    II.1]{BorelJi}) A partition of $\cM_1(M)$ of the stated type does
  exist, as in \eqref{Eqcores.108}. Proposition~\ref{Eqcores.109} shows the
  existence of a $G$-invariant product-type metric, collective boundary
  defining functions and product decompositions near the boundary
  hypersurfaces. First consider the union of two copies of $M,$ denoted
  $M^\pm,$ with all points in $\supp(B_1),$ i\@.e\@.~all points in the
  boundary hypersurfaces in $B_1,$ identified
\begin{equation}
M_1=(M^+\sqcup M^-)\slash\simeq_1,\ p\simeq_1p',\ p=p'\Min H\in B_1.
\label{Eqcores.113}\end{equation}
Now, the local product decompositions near each element of $B_1$ induce a
\ci\ structure on $M_1$ making it again a manifold with corners. Thus
$\rho_1,$ the collective defining function for $B_1$ on $M=M^+$ can be
extended to the smooth function 
\begin{equation}
\rho'_1=\begin{cases}\rho_1&\Mon M^+\\
-\rho_1&\Mon M^-.
\end{cases}
\label{Eqcores.114}\end{equation}
Similarly the corresponding normal vector field $V_1$ extends to be smooth
when defined as $-V_1$ on $M^-.$ The action of $G$ on $M$ gives actions on
$M^\pm$ which are consistent on $\supp(B_1)$ and the product decomposition
of the group action shows that the combined action on $M_1$ is smooth. The
boundary hypersurfaces of $M_1$ fall into two classes. Those arising from
boundary hypersurfaces of $M$ which meet one of the elements of $B_1,$
these appear as the doubles of the corresponding hypersurfaces from $M.$
The boundary hypersurfaces of $M$ which do not meet an element of $B_1$
contribute two disjoint boundary hypersurfaces to $M_1.$ It follows that
the decomposition of $\cM_1(M)$ in \eqref{Eqcores.108} induces a similar
decomposition of $\cM_1(M_1)$ in which each $B_i,$ $i=2,\dots,l$ contains
the preimages of the boundary hypersurfaces of $M,$ other than the elements
of $B_1,$ under the natural projection $M_1\longrightarrow M.$ The $\bbZ_2$
action on $M_1$ given by exchanging signs is smooth, by construction, and
commutes with the $G$-action.

Thus this procedure can be repeated $l$ times finally giving a manifold
without boundary with smooth $G$-action as desired.
\end{proof}

\section{Invariant tubes and collars}\label{Collar}

As note above the doubling construction allows the standard
properties of group actions on boundaryless manifolds to be transferred to 
the context of manifolds with corners. In fact the standard proofs may also
be extended directly.

If $\zeta \in M$ then the stabilizer $G_\zeta$ acts on $T_\zeta M$ and on
the metric balls, of an invariant product-type metric, in $T_\zeta M.$ If
$\zeta$ is contained in a corner of codimension $k \geq 0,$ then the
exponential map for the metric identifies the set of the inward-pointing
vectors in a small ball in $T_\zeta M$ with a $G_\zeta$-invariant
neighborhood of $\zeta$ in $M$ hence establishes the basic linearization
result.

\begin{proposition}[Bochner] \label{Bochner} If $\zeta \in M$ is contained
  in a corner of codimension $k \geq 0$ then there is a $G_\zeta$-invariant
  neighborhood $\cU_\zeta $ of $\zeta$ in $M,$ a linear action $\alpha
  _\zeta$ of $G_\zeta$ on $\bbR^{m,k},$ and a $G_\zeta$-equivariant
  diffeomorphism $\chi_\zeta : \cU_\zeta \longrightarrow B^+ $ to (the
  inward-pointing part of) a ball $B^+\subset\bbR^{m,k}.$
\end{proposition}

\begin{corollary}\label{FixedPtSet} If $G$ is a compact Lie group acting
  smoothly on a manifold $M,$ then $M^G = \{ \zeta \in M;g\cdot \zeta =
  \zeta \Mforall g \in G \}$ is an interior $p$-submanifold of $M.$
\end{corollary}

A \emph{slice at $\zeta \in M$} for the smooth action of $G$ is a p-submanifold,
$S,$ of $M$ through $\zeta$ such that
\begin{quote}
i) $T_\zeta M = \alpha_{\zeta}(\df g) \oplus T_\zeta S,$ \\
ii) $T_{\zeta'} M = \alpha_{\zeta'}(\df g) + T_{\zeta'} S$
for all $\zeta' \in S,$ \\
iii) $S$ is $G_\zeta$-invariant, \\
iv) If $g \in G$ and $\zeta' \in S$ are such that $g\cdot \zeta' \in S$
then $g \in G_{\zeta}.$
\end{quote}

For $\eps \in (0,1),$ set
\begin{equation*}
S_\eps = \chi^{-1}_\zeta ( \alpha_\zeta(\df g)^{\perp} \cap B^+(\eps))
\end{equation*}
where $B^+(\epsilon)\subset T_{\zeta}M$ is the set of inward-pointing
vectors of length less than $\eps.$ Since the vector fields in the image of
$\alpha$ are tangent to all of the boundary faces, $S_\eps$ is necessarily a
p-submanifold of $M$ through $\zeta.$ Elements $k \in G_\zeta$ satisfy
$T_\zeta A(k)(\alpha_\zeta(X)) = \alpha_\zeta(\Ad k(X)),$ so the tangent
action of $G_\zeta$ preserves $\alpha_\zeta(\df g)$ and hence $S_\eps$ is
$G_\zeta$-invariant. The Slice Theorem for boundaryless manifolds
\cite[Theorem 2.3.3]{Duistermaat-Kolk}, applied to $\double M,$ shows that
$S_\eps$ is a slice for the $G$-action at $\zeta$ if $\eps$ is small
enough.

Similarly, the following result is \cite[Theorem 2.4.1]{Duistermaat-Kolk}
applied to $\double M.$

\begin{proposition}[Tube Theorem] \label{TubeThm} If $G$ acts smoothly on a
  manifold $M$ and $\zeta \in M,$ then there is a representation space $V$
  of $G_\zeta$ with $G_\zeta$-invariant subset $V^+$ of the form
  $\bbR^{\ell,k},$ a $G$-invariant neighborhood $U$ of $\zeta \in M,$ a
  $G_\zeta$-invariant neighborhood, $V,$ of the origin in $V^+$ and a
  $G$-equivariant diffeomorphism
\begin{equation*}
\phi: G \times_{G_\zeta} V^+\longrightarrow U\ \Mst \phi(0) = \zeta.
\end{equation*}
\end{proposition}

It is straightforward to check (see \cite[Lemma 4.16]{Kawakubo}) that the
$G$-isotropy group of $[(g, v)] \in G \times_{G_\zeta} V$ is conjugate (in
$G$) to the $G_\zeta$ isotropy group of $v$ in $V.$ Thus, if $U$ is a
neighborhood of $\zeta$ as in Proposition \ref{TubeThm} and $\zeta' \in U,$
then
\begin{equation}\label{IsotropyShrink}
G_{\zeta'} \text{ is conjugate to a subgroup of } G_{\zeta}.
\end{equation}

Exponentiation using a product-type $G$-invariant metric allows a
neighborhood of a $G$-invariant p-submanifold $X\subseteq M$ to be
identified with a neighborhood of the zero section of its normal bundle.

\begin{proposition}[Collar Theorem] \label{ColThm} If $G$ acts smoothly on
  a manifold $M$ and $X \subseteq M$ is a $G$-invariant interior
  p-submanifold, then there exists a $G$-invariant neighborhood $U$ of $X$
  in $M$ and a $G$-invariant diffeomorphism from the normal bundle $NX$ of
  $X$ to $U$ that identifies the zero section of $NX$ with $X$ and for all
  sufficiently small $\eps>0$ the submanifolds
\begin{equation*}
	\cS_\eps(X) = \{ \zeta \in M; d(\zeta, X) = \eps \}
\end{equation*}
are $G$-invariant and the $G$-actions on $S_\eps(X)$ and $S_{\eps'}(X)$
are intertwined by translation along geodesics normal to $X.$
\end{proposition}

\begin{proof} As a p-submanifold, $X$ has a tubular neighborhood in $M,$
  which by exponentiating we can identify with
\begin{equation}\label{Ueps}
	U_\eps = \{\zeta \in M; d(\zeta, X) \leq \eps \}.
\end{equation}

For $\eps$ small enough, each $\zeta \in U_\eps$ is connected to $X$ by a
unique geodesic of length less than $\eps,$ $\gamma_\zeta.$ Since the
$G$-action is distance preserving and short geodesics are the unique
length-minimizing curves between their end-points,
\begin{equation}
g\cdot \gamma_{\zeta} = \gamma_{g\cdot \zeta}, \Mforevery g \in G,\
\zeta \in U_{\eps}.
\label{Resolution.July13.1}\end{equation}
It follows that $G$ preserves $S_{\eps'}(X)$ for all $\eps' <\eps$ and
that translation along geodesics normal to $X$ intertwines the
corresponding $G$-actions, as claimed.
\end{proof}

If $\Phi: M \longrightarrow Y$ is an equivariant fibration and $G$ acts
trivially on $Y,$ we can find a $G$-invariant submersion metric.
Exponentiating from a fiber of $\Phi,$ $\Phi^{-1}(q),$ gives an equivariant
identification with a neighborhood of the form $\Phi^{-1}(q) \times U,$ which
establishes the following result.

\begin{proposition}
Suppose $G$ acts on the manifolds $M$ and $Y,$ $\Phi:M \longrightarrow Y$ 
is an equivariant fibration, and the action of $G$ on $Y$ is trivial, 
then the fibers of $\Phi$ are $G$-equivariantly diffeomorphic.
\label{Eqcores.F.1}\end{proposition}

\section{Boundary resolution}\label{sec:BounRes}

In this section the first steps towards resolution of a group action by
radial blow-up are taken. Namely it is shown that on the blow-up of a
$G$-invariant closed p-submanifold, $X,$ the group action extends smoothly,
and hence uniquely, from $M\setminus X$ to $[M;X];$ the blow-down map is
then equivariant. Using this it is then shown that any smooth action, not
requiring \eqref{Eqcores.108}, on a manifold with corners lifts to a
boundary intersection free action, i\@.e\@.~one which does satisfy
\eqref{Eqcores.108}, after blowing-up appropriate boundary faces.

Let $\cJ(M)$ be the set of isotropy groups for a smooth action of $G$ on $M.$ 

\begin{proposition} \label{EquivBlowup} If $X \subseteq M$ is a
  $G$-invariant closed p-submanifold for a smooth action by a compact Lie
  group, $G,$ on $M$ then $[M;X]$ has a unique smooth $G$-action such that
  the blow-down map $\beta:[M;X]\longrightarrow M$ is equivariant and
\begin{equation}
\cJ([M;X])=\cJ(M \setminus X).
\label{Eqcores.223}\end{equation}
\end{proposition}

\begin{proof}
The blown-up manifold is
\begin{equation*}
	[M;X] = N^+X \sqcup (M \setminus X)
\end{equation*}
with smooth structure consistent with the blow up of the normal bundle to
$X$ along its zero section. Thus $[M;X]$ is diffeomorphic to $M \setminus U_\eps$ with
$U_{\eps}$ as in \eqref{Ueps}. This diffeomorphism induces a smooth
$G$-action on $[M;X]$ with respect to which the blow-down map is equivariant.
The result for isotropy groups, \eqref{Eqcores.223}, follows
from \eqref{IsotropyShrink}, namely the isotropy groups away from the front
face of $[M;X]$ are certainly identified with those in $M\setminus X$ and
the isotropy groups on $\ff([M;X])$ are identified with those in $S_\eps$
for small $\eps>0.$
\end{proof}

A general smooth group action will lift to be boundary intersection free
on the {\em total boundary blow-up} of $M.$ This manifold $M_{\tb},$
discussed in \cite[\S2.6]{Hassell-Mazzeo-Melrose}, is obtained from $M$ by
blowing-up all of its boundary faces, in order of increasing
dimension. Blowing up all of the faces of dimension less than $k$ separates
all of the faces of dimension $k$ so these can be blown-up in any order
without changing the final space which is therefore well-defined up to
canonical diffeomorphism.

In Figure \ref{fig:BdyFree} the octagon is obtained from the square by 
blowing-up the corners and the $\bbZ/4$-action lifts from the square to the 
boundary intersection free action on the octagon.

\begin{proposition}\label{BounRes} If $G$ acts smoothly on a manifold $M,$
  without necessarily satisfying \eqref{Eqcores.108}, the induced action of
  $G$ on $M_{\tb}$ is boundary intersection free, i\@.e\@. does satisfy
  \eqref{Eqcores.108}.
\end{proposition}

\begin{proof} Let $\beta: M_{\tb}\longrightarrow M$ be the blow-down
  map. Any boundary hypersurface $Y$ of $M_{\tb}$ is the lift of a boundary
  face $F= \beta(Y)$ of $M.$ Since each element $G$ acts on $M$ by a
  diffeomorphism it sends $\beta(Y)$ to a boundary face of $M$ of the same
  dimension as $F,$ say $F'= \beta(Y').$ The induced action on $M_{\tb}$
  will send the boundary face $Y$ to $Y'$ and, from the definition of
  $M_{\tb},$ $Y'$ is either equal to $Y$ or disjoint from $Y.$ Hence the
  action of $G$ on $M_{\tb}$ is boundary intersection free.  \end{proof}

In fact it is generally possible to resolve an action to be boundary
intersection free by blowing up a smaller collection of boundary
faces. Namely, consider all the boundary faces which have the property that
they are a component of an intersection $H_1\cap\dots\cap H_N$ where the
$H_i\in\cM_1(M)$ are intertwined by $G,$ meaning that for each $1\le i<j\le
N$ there is an element $g_{ij}\in G$ such that $g_{ij}(H_j)=H_i.$ This
collection of boundary faces satisfies the chain condition that if $F$ is
an element and $F'\supset F$ then $F'$ is also an element. In fact this
collection of boundary faces is divided into transversal subcollections
which are closed under intersection and as a result the manifold obtained
by blowing them up in order of increasing dimension is well-defined. It is
straightforward to check that the lift of the $G$-action to this partially
boundary-resolved manifold is boundary intersection free.

\section{Resolution of $G$-actions}\label{sec:Resolution}

The set, $\cJ(M),$ of isotropy groups which occur in a smooth $G$-action is
necessarily closed under conjugation, since if $G_\zeta\in\cJ$ then
$G_{g\zeta}=gG_\zeta g^{-1}.$ Let $\cI=\cJ/G$ be the set of conjugacy
classes of isotropy groups for the action of $G$ on $M$ and for each
$I\in\cI$ let
\begin{equation}
M^I=\{\zeta\in M;G_{\zeta}\in I\},
\label{Eqcores.224}\end{equation}
be the corresponding \emph{isotropy type.} Proposition~\ref{TubeThm} shows
these to be smooth p-submanifolds and they stratify $M,$ with a natural
partial order
\begin{equation*}
I' \preccurlyeq I\Mor M^I\preccurlyeq M^{I'}\Mif 
K\in I\text{ is conjugate to a subgroup of an element of }I'.
\end{equation*}
Thus minimal elements with respect to $\preccurlyeq$ are the ones with the
largest isotropy groups. We also set 
\begin{equation}
M_I=\clos(M^I)\subset\bigcup_{I'\preccurlyeq I} M^{I'}
\label{Eqcores.225}\end{equation}

\begin{proposition}\label{Eqcores.229} The isotropy types $M^I\subset M$
  for a smooth action by a compact group $G$ form a finite collection of
  p-submanifolds each with finitely many components.
\end{proposition}

\begin{proof} In \cite[Proposition 2.7.1]{Duistermaat-Kolk}, this result is
  shown for boundaryless manifolds. By passing from $M$ to $\double{M}$ as
  in Theorem \ref{DoublingThm}, the same is true for manifolds with
  corners with the local product condition implying that $M^I$ is a
  p-submanifold following from Proposition~\ref{FixedPtSet}.
\end{proof}

\begin{definition}\label{Eqcores.226} A \emph{resolution} of a smooth
  $G$-action on a compact manifold $M$ (with corners) is a manifold, $Y,$
  obtained by the successive blow up of closed $G$-invariant p-submanifolds
  of $M$ to which the $G$-action lifts to have a unique isotropy type.
\end{definition}
\noindent Proposition~\ref{EquivBlowup} shows that there is a unique lifted
$G$-action such that the iterated blow-down map is $G$-equivariant.

Such a resolution is certainly not unique -- as in the preceding section,
in the case of manifolds with corners, it is always possible to blow up a
boundary face in this way, but this is never required for the resolution
of an action satisfying \eqref{Eqcores.108}. We show below that there is a
canonical resolution obtained by successively blowing up minimal isotropy
types. To do this we note that the blow-ups carry additional structure.

\begin{definition}\label{Eqcores.227} An \emph{equivariant resolution
    structure} for a $G$ action on a manifold $Y$ is a resolution
  structure, in the sense of Definition~\ref{Eqcores.121}, with
  $G$-equivariant fibrations to bases each with unique isotropy type and
  such that in addition none of the isotropy types in any base is present
  in the total space. A \emph{full resolution} for a $G$-action on a
  manifold, $M,$ is a resolution in the sense of
  Definition~\ref{Eqcores.226} where $Y$ carries such an equivariant
  resolution structure.
\end{definition}

\begin{proposition}\label{Eqcores.228} Let $M$ be a smooth manifold with a
  smooth boundary intersection free action by a compact Lie group $G$ and
  an equivariant resolution structure, then any minimal isotropy type $X=M^I$ is a
  closed interior p-submanifold and if it is transversal to the fibers of
  all the boundary fibrations then $[M;X]$ has an induced equivariant
  resolution structure.
\end{proposition}

\begin{proof} As for a boundaryless manifold the minimal isotropy type is 
  closed in $M$ since its closure can only contain points with larger
  isotropy group. It is an interior p-submanifold by
  Proposition~\ref{Eqcores.229}, thus the blow up $[M;X]$ is
  well-defined. The $G$-action lifts smoothly to $[M;X]$ by
  Proposition~\ref{EquivBlowup} and the defining isotropy type $I$ is not
  present in the resolved action. The assumed transversality allows
  Proposition~\ref{Eqcores.126} i) to be applied to conclude that the
  resolution structure lifts to $[M;X]$ and so gives an equivariant
  resolution structure.
\end{proof}

\begin{theorem}\label{thm:ResStructure} A compact manifold (with corners),
  $M,$ with a smooth, boundary intersection free, action by a compact Lie
  group, $G,$ has a canonical full resolution, $Y(M),$ obtained by
  iterative blow-up of minimal isotropy types.
\end{theorem}

\begin{proof} In view of Proposition~\ref{Eqcores.228} it only remains to
  show, iteratively, that at each stage of the resolution any minimal
  isotropy type is transversal to the fibers of the earlier blow ups.

At the first step the transversality condition is trivial, since there is
no boundary, and so the first blow-up can be carried out and leads to an
equivariant resolution structure. Thus we can assume, inductively, that the
equivariant resolution structure exists at some level and then we simply need
to check that any minimal isotropy type for the lifted action is
transversal to the fibers of each of the fibrations. Transversality is a
local condition and at a point of boundary codimension greater than one the
compatibility condition for a resolution structure ensures that
the fibration of one of the boundary hypersurfaces through that point has
smallest leaves and it is necessarily the `most recent' blow up. Thus we
need only consider the case of a point of intersection of the minimal
isotropy type and the front face produced by the blow up of an earlier
minimal isotropy type in which there are (locally) no intermediate blow
ups. Working locally, in the manifold before the earlier of the two blow
ups, we simply have a manifold with a $G$-action and two intersecting
isotropy types, one of which is locally minimal.

Now, by Proposition \ref{TubeThm}, if $\zeta$ is such a point of
intersection, with isotropy group $H,$ it has a neighborhood, $U,$ with a
$G$-equivariant diffeomorphism to $L= G \times_{H} V^+$ with $V^+$ the
inward-pointing unit ball in a representation space $V$ for $H.$ The points
in $V$ with isotropy group $H$ form a linear subspace and $H$ acts on the
quotient. Thus the action is locally equivariantly diffeomorphic to
$G\times_H W^+\times B$ where the action is trivial on $B$ and $W^+\subset
W$ is a ball around the origin in a vector space $W$ with linear $H$-action
such that $W^H=\{0\}.$ Thus any isotropy type meeting $M^H$ at $\zeta$ is
represented as a twisted product by $G\times_{H} (V^+)^{I}\times B$ where $I$
is an isotropy class in $H.$ In particular such a neighboring isotropy type
is a bundle over the minimal isotropy type and meets the fibers of a normal
sphere bundle of small radius transversally. Thus, on blow-up it meets the
fibers of the front face, which are these spheres, transversally.

Thus in fact the successive blow-ups are always transversal to the fibers
of the early ones and hence the successive partial resolution structures
lift and finally give a full resolution.

The uniqueness of this full resolution follows from the fact that at each
stage the alternative is to blow up one of a finite set of
minimal isotropy types. Since these are disjoint the order at this stage
does not matter and hence, inductively, any such order produces a
canonically diffeomorphic full resolution.
\end{proof}

\begin{remark} \label{resolution.2}
As mentioned in the introduction, there is a one-to-one correspondence 
between the isotropy types of the $G$-action on $M$ and the $G$-invariant 
collective boundary hypersurfaces in the resolution structure of $Y(M).$
The base of the boundary fibration corresponding to the isotropy type 
$M^{[K]}$ is the canonical resolution of $M_{[K]},$ the closure of $M^{[K]}$ 
in $M,$ i.e., $Y_{[K]}(M) = Y(M_{[K]}).$ 
If $M$ is connected, the isotropy type of $Y(M)$ is the unique open, or principal,
isotropy type of $M,$ so $M$ and $Y(M)$ can be thought of as different
compactifications of the same open set. 
\end{remark}

Consider the action of $\bbS^1$ on $\bbS^2$ by rotation around the $z$-axis.
There are two isotropy types: one consisting of the `north pole' and `south 
pole', $\{ N, S \},$ has isotropy group $\bbS^1,$ while the complement has 
isotropy group $\{\Id\}.$ The resolution is obtained by blowing-up 
the former isotropy type and keeping the blow-down maps as the boundary 
fibration,
\begin{equation*}
	Y(\bbS^2) = [\bbS^2; \{ N, S \} ].
\end{equation*}
\begin{figure}[htpb]
\centering
\includegraphics{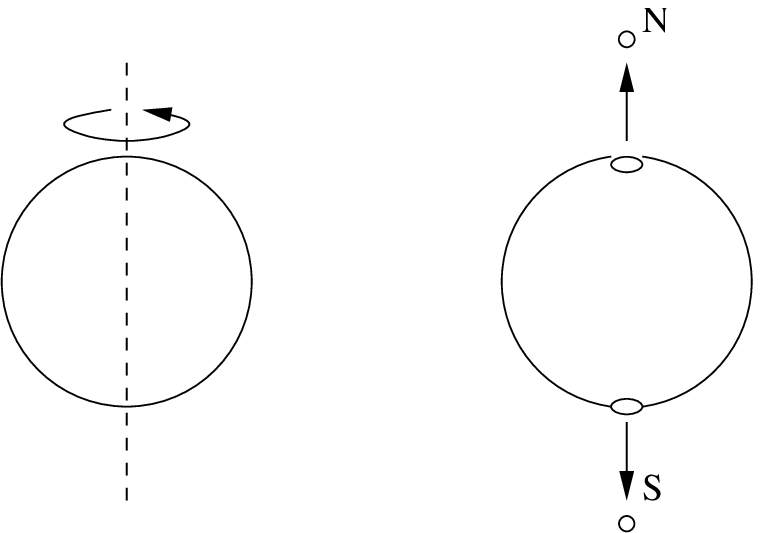}
\caption{The resolution of the $\bbS^1$ action on $\bbS^2.$}
\end{figure}

In non-trivial cases, the resolution of a product of $G$-actions is not
equal to the product of the resolutions. For instance, consider the
$\bbZ_2$-action on $[-1,1]$ given by reflecting across the origin and the
product action of $\bbZ_2 \times \bbZ_2$ on $[-1,1] \times [-1,1].$ The
resolution of $[-1,1]$ is
\begin{equation*}
	Y([-1,1]) = \big[ [-1,1]; \{ 0\} \big] = [-1,0] \sqcup [0,1]
\end{equation*}
while the resolution of $\bbR^2$ is
\begin{equation*}
	Y(\bbR^2) = \big[ [-1,1]^2; \{ (0,0) \}; [-1,1] \times \{ 0 \}; \{0 \} \times [-1,1] \big]
\end{equation*}
which in particular is not equal to $Y([-1,1])^2.$ 
\begin{figure}[htpb]
\centering
\label{ProductFigure}
\includegraphics{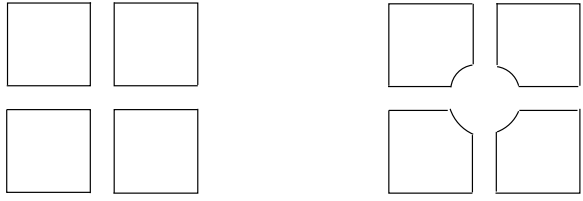}
\caption{With these actions, $Y([-1,1])^2$ is the disjoint union of four
  quadrants and $Y([-1,1]^2)$ is the disjoint union of four quadrants with a
  corner blown-up.} 
\end{figure}

\section{Resolution of orbit spaces}\label{Quot}

Having constructed a resolution of the group action, we now view the consequences
for the orbit space. For boundaryless manifolds with a unique isotropy type Borel 
showed that the orbit space is a smooth 
manifold, and the natural projection onto it is a smooth fibration, though in the 
non-free case not a principal bundle. It is straight-forward to extend this to manifolds 
with corners.

\begin{proposition}[Borel] \label{lem:Borel} Let $M$ be a manifold with a
  (boundary intersection free) $G$-action with a unique isotropy type, if
  $N(K)$ is the normalizer of an isotropy group $K$ then $M$ is
  $G$-equivariantly diffeomorphic to $G \times_{N(K)} M^K$ and the
  inclusion $M^K \hookrightarrow M$ induces a diffeomorphism
\begin{equation*}
(N(K)/K)\backslash M^K = N(K)\backslash M^K \cong G\backslash M.
\end{equation*}
\end{proposition}

\begin{proof} We follow the proof of \cite[Theorem 2.6.7]{Duistermaat-Kolk}
in the boundaryless case. It is shown in Corollary~\ref{FixedPtSet} that
for a fixed isotropy group $M^K$ is a smooth interior
p-submanifold. The normalizer $N(K)$ acts on $M^K$ with isotropy group $K$ 
so the quotient group $W(K)=N(K)/K$ acts freely on $M^K.$ Thus the quotient
$W(K)\backslash M^K$ is smooth. The diagonal action of $N(K)$ on the product 
\begin{equation}
N(K)\times(G\times M^K)\ni (n,(g,m))\longmapsto (gn^{-1},nm) \in G \times M^K
\label{Eqcores.208}\end{equation}
is free, so the quotient $G\times_{N(K)}M^K$ is also smooth. Moreover the
action of $G$ on $M$ factors through the quotient, $gm=gn^{-1}\cdot nm,$ so
defines the desired smooth map
\begin{equation}
G\times_{N(K)}M^K\longrightarrow M.
\label{Eqcores.209}\end{equation}
This is clearly $G$-equivariant for the left action of $G$ on
$G\times_{N(K)}M^K$ and is the identity on the image of $\{\Id\}\times M^K$
to $M^K.$ The Slice Theorem shows that the inverse map, $m\longmapsto
[(g,m')]$ if $m'\in M^K$ and $gm'=m$ is also smooth, so \eqref{Eqcores.209}
is a $G$-equivariant diffeomorphism.

The quotient $G \backslash(G\times_{N(K)}M^K)=N(K)\backslash M^K$
is smooth and the smooth structure induced on $G \backslash M$ is
independent of the choice of $K.$
\end{proof}

Thus if $Y(M)$ is the canonical resolution of the $G$-action on $M,$ the orbit
space
\begin{equation*}
	Z(M) = G \backslash Y(M)
\end{equation*}
is a smooth manifold with corners, as is the orbit space, $Z_I,$ of each $Y_I.$
Moreover, the boundary hypersurfaces of $Z(M)$ may be identified with the 
equivalence classes under the action of $G$ of the boundary hypersurfaces of 
$Y(M)$ and the boundary fibrations of $Y(M),$ being $G$-equivariant, descend 
to give a resolution structure on $Z(M).$

\section{Equivariant maps and resolution} \label{sec:Maps}

Given two manifolds with $G$-actions and an equivariant map between them, there
need not be a corresponding map between their canonical resolutions. Any map can
be factored into the composition of an embedding followed by a fibration, and in this 
section we describe the relation between these maps and resolution. In particular we
will discuss the resolution of a space with respect to an equivariant fibration.

The behavior of resolution with respect to embeddings is particularly simple.

\begin{theorem}
Let $X$ and $M$ be manifolds with $G$-actions, and let $i: X \hookrightarrow M$ 
be an equivariant embedding of $X$ as a p-submanifold of $M.$ Let $[K]$ be the 
open isotropy type of $X,$ so that $i(X) \subseteq M_{[K]},$ and let
\begin{equation*}
	Y_{[K]}(M) \xrightarrow{\text{ }\beta_{[K]}\text{ }} M_{[K]}
\end{equation*}
be the resolution of $M_{[K]}.$ Then 
\begin{equation*}
	Y(X) = \overline{ \beta_{[K]}^{-1}( X_{[K]} ) },
\end{equation*}
where the closure is taken in $Y_{[K]}(M),$ and so we have a commutative diagram
\begin{equation*}
\xymatrix{Y(X) \ar@{^(->}[r] \ar[d]^{\beta_X} & Y_{[K]}(M) \ar[d]^{\beta_{[K]}} \\ 
		X \ar@{^(->}[r]^{i} &  Y }
\end{equation*}
\end{theorem}

\begin{proof}
It suffices to note that, if $S,$ $X\subset M$ are closed p-submanifolds
with a common local product description, so $S\cap X\subset X$ is a
p-submanifold and if $\gamma:[M;S]\longrightarrow M$ is the blow-down map, then
\begin{equation*}
\overline{ \gamma^{-1}( X\setminus S ) } = [X; X \cap S].
\end{equation*}
Thus, since $X_{[L]} = M_{[L]} \cap X$ for every subgroup $L$ of $G,$ resolving
$M_{[K]}$ simultaneously resolves $X = X^{[K]}.$ 
\end{proof}

We next consider the resolution of the total space of a fibration, first
without group actions.

Suppose that $X$ and $M$ carry resolution structures and $f:X
\longrightarrow M$ is a fibration with the property that, for each $H \in
\cM_1(X)$ such that $f(H) \in \cM_1(X),$ $f$ maps the fibers of $\phi_H$
to the fibers of $\phi_{f(H)}.$ Thus $f$ induces a fibration $\bar f_H: Y_H
\longrightarrow Y_{f(H)}$ covered by $f\rest{H}$. In this case we say that
$f$ is a {\em resolution fibration}. Note that for $H\in\cM_1(X),$ $f(H)$
is either a boundary hypersurface of $M$ or a component of $M.$

\begin{lemma} \label{BlowUpLemmaItFib} Suppose $X$ and $M$ carry resolution
  structures and $f:X \longrightarrow M$ is a resolution fibration.
\begin{itemize}
\item [i)] If $S \subseteq X$ is a closed p-submanifold transverse to the
  fibers of $f$ and to the  fibers of $\phi_H$ for each $H \in \cM_1(X),$
  then the composition of $f$ with the blow-down map $[X;S] \longrightarrow
  X \longrightarrow M$ is a resolution fibration.
\item [ii)] If  $T \subseteq M$ is a closed interior p-submanifold
  transverse to the fibers of each $\phi_K$ for $K \in \cM_1(M),$ then $f$
  lifts from $X \setminus f^{-1}(T)$ to a resolution fibration $[X;
    f^{-1}(T)] \longrightarrow [M;T].$
\item [iii)] If $L \subseteq Y_H$ is an interior p-submanifold (with $\dim
  L < \dim Y_H$) for some $H \in \cM_1(X)$ and $\phi_H^{-1}(L)$ is
  transverse to the fibers of $f$ then the composition, $[X;\phi_H^{-1}(L)]
  \longrightarrow X \longrightarrow M,$ of $f$ with the blow-down map is a
  resolution fibration.
\end{itemize}
\end{lemma}

\begin{proof} For i), it only remains to establish that $[X;S]
  \longrightarrow X \longrightarrow M$ is a resolution fibration. If $H \in
  \cM_1(X)$ intersects $S$ then $\phi_H$ restricts to $S \cap H$ to a
  fibration $\xymatrix@1{Z_S\ar@{-}[r]&S \cap H\ar[r]& Y_H}$ and on passing
  from $X$ to $[X;S]$ the boundary fibration $\xymatrix@1{Z\ar@{-}[r]&
    H\ar[r]& Y_H}$ is replaced with $\xymatrix@1{[Z;Z_S]\ar@{-}[r]& [H; H
      \cap S] \ar[r]& Y}.$  Thus the blow-down map sends the fibers of 
  each boundary fibration of $[X;S]$ to a fiber of a boundary
  fibration of $X$, and so the composition with $f$ is a resolution fibration.

For ii), first note that $f^{-1}(T)$ is transversal to the fibers of each
$\phi_H$ for $H \in \cM_1(X)$ because $f$ is a resolution fibration, and
hence $[X; f^{-1}(T)]$ has a resolution structure. That the
lift of $f$ is a resolution fibration follows as in i).

Finally, for iii), since $\phi_H^{-1}(Z)$ is transverse to the fibers of
$f,$ the composition $[X;\phi_H^{-1}(Z)] \xrightarrow{\beta} X
\xrightarrow{f} M$ is a fibration by Lemma \ref{BlowUpLemmaFib} i) and a
resolution fibration by the same argument as in i).
\end{proof}

Even an equivariant fibration between two manifolds with smooth actions by
the same group does not in general lift to a smooth map between their
canonical resolutions. However there is a natural resolution of the total
space relative to the fibration to which it lifts to fibration to the
canonical resolution of the base.

\begin{theorem} If $f:X\longrightarrow M$ is an equivariant fibration
  between compact manifolds with smooth actions by a compact Lie group $G,$
then there is a natural full resolution of the action on $X,$ denoted $Y(X,f),$
such that $f$ lifts to a fibration giving a commutative diagram
\begin{equation*}
\xymatrix{ Y(X,f) \ar[r]^{Y(f)} \ar[d]  & Y(M) \ar[d] \\ X \ar[r]^{f} & M.}
\end{equation*}
\end{theorem}
Thus relative resolution of the total space \emph{does} depend, in general, on the
fibration. It reduces to the canonical resolution when the fibration is
trivial in the sense that it is either the map to a point or the identity map.

\begin{proof} The first step in the construction of $Y(X,f)$ is to carry out the
canonical resolution of $M$ to $Y(M).$ We proceed by induction, assuming
both $X$ and $M$ carry resolution structures and that $f$ is an equivariant
resolution fibration, as discussed above. The inductive step is to blow up
a minimal isotropy type in $M.$ The inverse image under the fibration $f$
  is a closed p-submanifold of $X$ and Lemma~\ref{BlowUpLemmaItFib} ii) shows
  that after blow of this submanifold $f$ lifts to an equivariant resolution
  fibration between the resolution structures on the blow-ups of $X$ and $M.$
  Thus the construction proceeds to give the canonical resolution of $M$ and
  a total space $X_0,$ with resolution structure and an equivariant
  resolution fibration
\begin{equation*}
	f_0: X_0 \longrightarrow Y(M).
\end{equation*}
In general $X_0$ is not a partial resolution of $X$ since the base spaces
of its boundary fibrations need not have unique isotropy group.

The second part of the procedure is to resolve $X_0,$ with its resolution
structure, without further change to $Y(M)$ and its resolution structure.
As in Remark~\ref{resolution.1}, $X_0 \longrightarrow Y(M)$ can also be
identified with the pull-back of $X \longrightarrow M$ along $Y(M) \longrightarrow
M.$ The boundary hypersurfaces of $X_0$ are in one-to-one correspondence
with the isotropy groups of $M.$ Thus for each isotropy type $M_{[K]}$ of 
$M,$ there are boundary hypersurfaces $P_{[K]} \subseteq X_0$ and $H_{[K]}
\subseteq Y(M)$ with fibrations $P_{[K]} \longrightarrow Q_{[K]}$ and
$H_{[K]} \longrightarrow Y_{[K]}$ forming a diagram
\begin{equation*}
\xymatrix{ P_{[K]} \ar[r]^{f_0} \ar[d] & H_{[K]} \ar[d]
\\
Q_{[K]} \ar[r]^{\bar f_{[K]}} & Y_{[K]}}
\end{equation*}
where all arrows are equivariant fibrations.  Moreover we may identify
$Y_{[K]}$ with $Y(M_{[K]})$ as in Remark~\ref{resolution.2} and $\bar
f_{[K]}: Q_{[K]} \longrightarrow Y_{[K]}$ with the pull-back of $f:
f^{-1}(M_{[K]}) \longrightarrow M_{[K]}$ along the map $Y(M_{[K]})
\longrightarrow M_{[K]}.$

\begin{lemma} \label{SimpleObs} Suppose $X$ is a manifold with a smooth
  $G$-action, $Y$ a manifold with a resolved $G$-action and $h:X
  \longrightarrow Y$ an equivariant fibration then $h\rest{S}:S
  \longrightarrow Y$ is surjective for any isotropy type $S \subseteq X.$
\end{lemma}

\begin{proof} If $\zeta \in S$ then $G_\zeta$ is necessarily a
  subgroup of $K = G_{f(\zeta)}.$ If $K$ is a normal subgroup of $G,$ then
  it acts trivially on $Y,$ and acts on each fibre of $h.$
  Proposition~\ref{Eqcores.F.1}) shows that the fibers of $h$ are
  $G$-equivariantly diffeomorphic, so an isotropy group that occurs in one
  fiber occurs in every fiber, and hence $h\rest{S}$ is surjective.

If $K$ is not a normal subgroup, then the fibration $h$ decomposes into fibrations
\begin{equation*}
	h^{K'}: h^{-1}(Y^{K'}) \longrightarrow Y^{K'}, \quad K' \in [K]
\end{equation*}
and, as each of these is surjective when restricted to $S \cap
h^{-1}(Y^{K'}),$ the result follows.
\end{proof}

Now, assuming that $X_0$ is not already a full resolution of $X,$ we
proceed to resolve it. Consider all the isotropy types of the action of $G$
on $X_0$ and the bases of its boundary fibrations and select one which is
minimal (and occurs in an unresolved component of $X_0$ or the base of its 
boundary fibrations). By the discussion in \S\ref{sec:Resolution} it is 
transverse to the resoution structure and Proposition~\ref{Eqcores.126} 
shows that $X_1,$ obtained by blowing it and all of its preimages in the 
resolution structure up, has a natural resolution structure. From 
Lemma~\ref{BlowUpLemmaItFib} i) and iii) the composition $X_1 
\longrightarrow X_0 \longrightarrow Y(M)$ is an equivariant resolution
fibration. In case the minimal istropy type occurs in $X_0$ itself, it is
an interior p-submanifold and the situation is even simpler. In all cases
the fibration $f$ lifts to a resolution fibration. Thus, after a finite
number of steps the total space and its resolution structure are also resolved.

As for the canonical resolution, the fact that the blow-ups are ordered
consistently with the partial order given by inclusion of isotropy groups,
and that minimal isotropy types are necessarily disjoint, ensures that this
construction of a relative resolution is also natural.
\end{proof}

\providecommand{\bysame}{\leavevmode\hbox to3em{\hrulefill}\thinspace}
\providecommand{\MR}{\relax\ifhmode\unskip\space\fi MR }
\providecommand{\MRhref}[2]{%
  \href{http://www.ams.org/mathscinet-getitem?mr=#1}{#2}
}
\providecommand{\href}[2]{#2}


\end{document}